\documentclass{amsart}
\usepackage{}

\usepackage{amsmath}
\usepackage{txfonts}
\usepackage{amsfonts}
\usepackage{hyperref}

\newtheorem{theorem}{Theorem}[section]
\newtheorem{lemma}[theorem]{Lemma}
\newtheorem{proposition}[theorem]{Proposition}
\newtheorem{corollary}[theorem]{Corollary}

\theoremstyle{definition}

\theoremstyle{remark}

\numberwithin{equation}{section}

\usepackage{geometry}
\begin{document}\large

\title{Harmonic 2-forms and positively curved 4-manifolds }

\author{Kefeng Liu}
\address{Mathematical Sciences Research Center, Chongqing University of Science and Technology, Chongqing
400054, China}

\address{Department of Mathematics, University of California at Los Angeles, California
90095, USA}

\curraddr{} \email{liu@math.ucla.edu}
\thanks{}

\author{Jianming Wan}

\address{School of Mathematics, Northwest University, Xi'an 710127, China}

\email{wanj\_m@aliyun.com}

\thanks{}

%    General info

\date{}

\dedicatory{}


\begin{abstract}\normalsize
We prove that if a compact Riemannian 4-manifold with positive sectional curvature satisfies a Kato type inequality, then it is definite. We also discuss some new insights for compact Riemannian 4-manifolds of positive sectional curvature.
\end{abstract}

\maketitle

\section{Introduction}
In this paper we focus on the discussion of compact oriented Riemannian 4-manifolds with positive sectional curvature. We begin with a famous conjecture addressed by Heinz Hopf:

\emph{ $S^{2}\times S^{2}$ does not admit a metric with positive sectional curvature.}

Until this day the conjecture is still widely open except for some special cases. For instance, see Berger \cite{[B1]}, Tasagas \cite{[T]}, Bourguignon \cite{[Bo]}, Seaman \cite{[S]}, Hsiang-Kleiner \cite{[HK]} and others. In \cite{[LL]} Li-Liu  proposed a geometric flow approach to this problem.

Let $M$ be a compact oriented 4-manifold. We denote by $b_{2}^{+}$ (resp. $b_{2}^{-}$) the number of positive (resp. negative) eigenvalues of the intersection form.
 $M$ is said to be definite if either $b_{2}^{+}=0$ or $b_{2}^{-}=0$. $S^{2}\times S^{2}$ is not definite since $b_{2}^{+}(S^{2}\times S^{2})=b_{2}^{-}(S^{2}\times S^{2})=1$.

We shall prove
\begin{theorem}\label{t1.1}
Let $M$ be a compact oriented Riemannian 4-manifold with sectional curvature $sec_{M}>0$. If any harmonic 2-form $\phi$ on $M$ satisfies
\begin{equation}\label{f1.1}
|\nabla\phi|^{2}\geq2|d|\phi||^{2},
\end{equation}
then $M$ is definite.
\end{theorem}

By Theorem \ref{t1.1}, to prove Hopf conjecture, one only needs to verify that formula (\ref{f1.1}) holds on compact oriented Riemannian 4-manifold with positive sectional curvature.

The proof of Theorem 1.1 is based on a Bochner formula of mixed type (Theorem 2.1), which is proved in Section 2. Formula (\ref{f1.1}) is closely related to Kato inequality. We will discuss the details in Section 4.

A general conjecture can be  formulated as: \emph{A compact simply connected 4-manifold with positive sectional curvature is homeomorphic (or diffeomorphic ) to $S^{4}$ or $\mathbb{C}\mathbb{P}^{2}$.}

We remark that the results in this paper were obtained in 2012. While trying to derive more complete result from the formulas in this paper, we have circulated this paper among many colleagues for the past several years.
%%%%%%%%%%%%%%%%%%%%%%%%%%%%%%%%%%%%%%%%%%%%%%%%%%%%%%%%%%%%%%%%%%%%%%%%%%%%%%%%%%%%%
\section{A Bochner type formula}

\subsection{}
 Let $M$ be a compact Riemannian manifold. Let $\Delta=d\delta+\delta d$ be the Hodge-Laplacian, where $d$ is the exterior differentiation and $\delta$
is the adjoint of $d$. Let $\phi\in\Omega^{k}(M)$ be a smooth $k$-form and $\{e_{i}, 1\leq i\leq n\}$ be an orthonormal frame. Then we have the well-known Weitzenb\"{o}ck formula
\begin{equation}\label{f2.1}
\Delta\phi=\sum_{i}\nabla^{2}_{e_{i}e_{i}}\phi-\sum_{i,j}\omega^{i}\wedge i(e_{j})R_{e_{i}e_{j}}\phi,
\end{equation}
where $\nabla^{2}_{XY}=\nabla_{X}\nabla_{Y}-\nabla_{\nabla_{X}Y}$ and $R_{XY}=-\nabla_{X}\nabla_{Y}+\nabla_{Y}\nabla_{X}+\nabla_{[X,Y]}$.
We say that a $k$-form $\phi$ is \emph{harmonic} if $\Delta\phi=0$. The famous Hodge theorem states that the de Rham cohomology $H^{k}_{dR}(M)$ is isomorphic to the space spanned by harmonic $k$-forms. Hodge theorem is a bridge to connect curvature and topology of manifolds.

Let $\phi=\sum_{i,j}\phi_{ij}\omega^{i}\wedge\omega^{j}$ be a harmonic 2-form, where $\{\omega^{i}, 1\leq i\leq n\}$ is the dual frame of $\{e_{i}, 1\leq i\leq n\}$ which is normal at $p$, i.e. $\nabla_{e_{i}}e_{j}(p)=0$.  By formula (\ref{f2.1}), we can obtain (c.f. \cite{[BS]} or \cite{[B]})
\begin{equation}\label{f2.2}
\Delta\phi_{ij}(p)=\sum_{k}(Ric_{ik}\phi_{kj}+Ric_{jk}\phi_{ik})-2\sum_{k,l}R_{ikjl}\phi_{kl},
\end{equation}
where $R_{ijkl}=\langle R(e_{i},e_{j})e_{k},e_{l}\rangle$ is the curvature tensor and $Ric_{ij}=\sum_{k}\langle R(e_{k},e_{i})e_{k},e_{j}\rangle$ is the Ricci tensor.
We will present a proof for (\ref{f2.2}) in the appendix.

Recall that the Hodge star operator is defined by
$$\ast( \omega^{1}\wedge\cdot\cdot\cdot\wedge\omega^{p})\coloneqq \omega^{p+1}\wedge\cdot\cdot\cdot\wedge\omega^{n}.$$
It can be extended linearly to all smooth forms. Obviously Hodge star operator maps harmonic forms to harmonic forms. A form $\phi$ is called self-dual if $\ast\phi=\phi$ and anti-self-dual if $\ast\phi=-\phi$. If $M$ is of 4k dimension, by Hodge theorem and de Rham theorem, we have $$b_{2k}^{+}=\dim\{\phi\in\Omega^{2k}(M): \ast\phi=\phi, \Delta\phi=0\}$$ and $$b_{2k}^{-}=\dim\{\phi\in\Omega^{2k}(M): \ast\phi=-\phi, \Delta\phi=0\}.$$ The signature $\sigma(M)=b_{2k}^{+}-b_{2k}^{-}$ and $b_{2k}=b_{2k}^{+}+b_{2k}^{-}$. So $M$ is definite if and only if $b_{2}(M)=|\sigma(M)|$.

It is also easy to see that $M$ is definite if and only if any harmonic 2k-form $\phi$ is self-dual or anti-self-dual.

\subsection{}
Let $M$ be a compact oriented Riemannian 4-manifold. Let $\phi$ be a harmonic 2-form on $M$. We denote by $$\phi_{+}=\phi+\ast\phi$$ and $$\phi_{-}=\phi-\ast\phi.$$
Then $\ast \phi_{+}=\phi_{+}$ and $\ast \phi_{-}=-\phi_{-}$. Set $F=\frac{1}{2}|\phi_{+}|^{2}$ and $G=\frac{1}{2}|\phi_{-}|^{2}$. We have the following Bochner type formula

\begin{theorem}
\begin{equation}\label{f2.3}
\Delta (FG)= 8KFG+G|\nabla\phi_{+}|^{2}+F|\nabla\phi_{-}|^{2}+2\langle\nabla F, \nabla G\rangle,
\end{equation}
where $K$ is the sum of some sectional curvatures determined by $\phi$. The explicit form of $K$ is given below (\ref{f2.7}).
\end{theorem}

 One of the interests of this formula is that the curvature term $K$ contains \emph{only} sectional curvature. This avoids the pinching assumption of sectional curvature, which is always necessary before.

For any $p\in M$, we can choose an orthonormal basis $\{e_{1}, e_{2},e_{3}, e_{4}\}$ of $T_{p}M$ such that
\begin{equation}\label{f2.4}
\varphi(p)=\lambda_{1}\omega^{1}\wedge \omega^{2}+\lambda_{2}\omega^{3}\wedge\omega^{4}.
\end{equation}
 That is to say, locally we can write $$\phi=f_{1}\omega^{1}\wedge \omega^{2}+f_{2}\omega^{3}\wedge \omega^{4}+f_{3}\omega^{1}\wedge \omega^{3}+f_{4}\omega^{1}\wedge \omega^{4}+f_{5}\omega^{2}\wedge \omega^{3}+f_{6}\omega^{2}\wedge \omega^{4}$$
and $$f_{1}(p)=\lambda_{1}, f_{2}(p)=\lambda_{2}, f_{3}(p)=f_{4}(p)=f_{5}(p)=f_{6}(p)=0.$$ Since any orthonormal basis at a point $p$ can be extended to an orthonormal frame normal at $p$. We may always assume that the frame field is normal at $p$.

Now we calculate the Laplacian of $F=\frac{1}{2}|\phi_{+}|^{2}=|\phi|^{2}+\ast(\phi\wedge\phi)$ and $G=\frac{1}{2}|\phi_{-}|^{2}=|\phi|^{2}-\ast(\phi\wedge\phi)$.
Note that under (\ref{f2.4}) we have $F(p)=(f_{1}+f_{2})^{2}(p)$ and $G(p)=(f_{1}-f_{2})^{2}(p)$.

\begin{lemma}\label{l2.1}
We have
\begin{equation}
\Delta f_{1}=2(Kf_{1}-R_{1234} f_{2})
\end{equation}
and
\begin{equation}
\Delta f_{2}=2(Kf_{2}-R_{1234} f_{1}),
\end{equation}
where
\begin{equation}\label{f2.7}
K=\frac{1}{2}(R_{1313}+R_{1414}+R_{2323}+R_{2424}).
\end{equation}
\end{lemma}

\begin{proof}
We write $\phi$ as $$\phi=\frac{f_{1}}{2}\omega^{1}\wedge \omega^{2}+\frac{-f_{1}}{2}\omega^{2}\wedge \omega^{1}+\frac{f_{2}}{2}\omega^{3}\wedge \omega^{4}+\frac{-f_{2}}{2}\omega^{4}\wedge \omega^{3}.$$
By formula (\ref{f2.2}), we obtain
$$\Delta\frac{f_{1}}{2}=\frac{f_{1}}{2}(Ric_{11}+Ric_{22})-(R_{1324}-R_{1423})f_{2}-R_{1212}f_{1}$$
and $$\Delta\frac{f_{2}}{2}=\frac{f_{2}}{2}(Ric_{33}+Ric_{44})-(R_{3142}-R_{3241})f_{1}-R_{3434}f_{2}.$$
Then the lemma follows from the  Bianchi identity $$R_{1234}+R_{1342}+R_{1423}=0.$$
\end{proof}

Now we will prove
\begin{proposition}\label{p2.3}
\begin{equation}
\Delta F=|\nabla\phi_{+}|^{2}+4(K-R_{1234})F
\end{equation}
and
\begin{equation}
\Delta G=|\nabla\phi_{-}|^{2}+4(K+R_{1234})G.
\end{equation}
\end{proposition}

\begin{proof}
From Lemma (\ref{l2.1}), one gets
\begin{equation}\label{f2.10}
f_{1}\Delta f_{1}+f_{2}\Delta f_{2} = 2K(f^{2}_{1}+f^{2}_{2})-4R_{1234}f_{1}f_{2}
\end{equation}
and
\begin{equation}\label{f2.11}
f_{1}\Delta f_{2}+f_{2}\Delta f_{1}=4Kf_{1}f_{2}-2R_{1234}(f^{2}_{1}+f^{2}_{2}).
\end{equation}
Combing (\ref{f2.10}), (\ref{f2.11}) and Green-Stokes formula, one obtains
\begin{eqnarray*}
\frac{1}{2}\Delta|\phi|^{2}& = &\frac{1}{2}\Delta(\sum_{i=1}^{6}f_{i}^{2})=\sum_{i=1}^{6}(f_{i}\Delta f_{i}+|\nabla f_{i}|^{2})\\
                              & = &f_{1}\Delta f_{1}+f_{2}\Delta f_{2}+\sum_{i=1}^{6}|\nabla f_{i}|^{2}\\
                              & = &\sum_{i=1}^{6}|\nabla f_{i}|^{2}+2K(f^{2}_{1}+f^{2}_{2})-4R_{1234}f_{1}f_{2}
\end{eqnarray*}
and
\begin{eqnarray*}
\frac{1}{2}\Delta[\ast(\phi\wedge\phi)]& = &\Delta(f_{1}f_{2})-\Delta(f_{3}f_{6})+\Delta(f_{4}f_{5})\\
                                             & = &2\langle\nabla f_{1},\nabla f_{2}\rangle-2\langle\nabla f_{3},\nabla f_{6}\rangle\\
                                             &&+2\langle\nabla f_{4},\nabla f_{5}\rangle+f_{1}\Delta f_{2}+f_{2}\Delta f_{1}\\
                                              & = &2\langle\nabla f_{1},\nabla f_{2}\rangle-2\langle\nabla f_{3},\nabla f_{6}\rangle+2\langle\nabla f_{4},\nabla f_{5}\rangle\\
                                              &&+4Kf_{1}f_{2}-2R_{1234}(f^{2}_{1}+f^{2}_{2}).
\end{eqnarray*}

Hence
\begin{eqnarray*}
\frac{1}{2}\Delta F& =&\frac{1}{2}\Delta|\phi|^{2}+\frac{1}{2}\Delta[\ast(\phi\wedge\phi)]\\
                   & =&|\nabla(f_{1}+f_{2})|^{2}+|\nabla(f_{3}-f_{6})|^{2}+|\nabla(f_{4}+f_{5})|^{2}\\
                   &&+2(K-R_{1234})(f_{1}+f_{2})^{2}\\
                   & = &\frac{1}{2}|\nabla(\phi+*\phi)|^{2}+2(K-R_{1234})F
\end{eqnarray*}
and
\begin{eqnarray*}
\frac{1}{2}\Delta G& =&\frac{1}{2}\Delta|\phi|^{2}-\frac{1}{2}\Delta[\ast(\phi\wedge\phi)]\\
                  & =&|\nabla(f_{1}-f_{2})|^{2}+|\nabla(f_{3}+f_{6})|^{2}+|\nabla(f_{4}-f_{5})|^{2}\\
                  &&+2(K+R_{1234})(f_{1}-f_{2})^{2}\\
                   & = & \frac{1}{2}|\nabla(\phi-*\phi)|^{2}+2(K+R_{1234})G.
\end{eqnarray*}

This completes the proof of the proposition.
\end{proof}

Theorem 2.1 follows from Proposition 2.3 directly.

%%%%%%%%%%%%%%%%%%%%%%%%%%%%%%%%%%%%%%%%%%%%%%%%%%%%%%%%%%%%%%%%%%%%%%%%%%%%%%%
\section{A proof of Theorem 1.1 and some vanishing results}

We use formula (\ref{f2.3}) to prove Theorem \ref{t1.1}. If $|\nabla\phi|^{2}\geq2|d|\phi||^{2}$ holds for any harmonic 2-form $\phi$. Then $$ |\nabla\phi_{+}|^{2}\geq\frac{|\nabla F|^{2}}{F}$$ and $$ |\nabla\phi_{-}|^{2}\geq\frac{|\nabla G|^{2}}{G}.$$
By using Schwarz inequality, one has $$G|\nabla\phi_{+}|^{2}+F|\nabla\phi_{-}|^{2}+2\langle\nabla F, \nabla G\rangle \geq 2|\nabla F|\cdot|\nabla G|+2\langle\nabla F, \nabla G\rangle\geq0.$$ By (\ref{f2.3}), one has $$\int_{M}KFGdv\leq0.$$ Since $K>0$, it follows that $F\equiv 0$ or $G\equiv 0$, i.e. $\phi$ is self-dual or anti-self-dual. Hence $M$ is definite.

\begin{corollary}
Let $M$ be a compact Riemannian 4-manifold with positive sectional curvature. If one of the following four conditions holds:\\
\textbf{1)} $F$ or $G$ is constant;\\
\textbf{2)} $FG$ is constant;\\
\textbf{3)} $F+G$ is constant;\\
\textbf{4)} $F-G$ is constant.\\
Then either $F\equiv0$ or $G\equiv0$, i.e. $\phi$ is self-dual or anti-self-dual.
\end{corollary}

\begin{proof}
\textbf{1)} follows directly from (\ref{f2.3}).

\textbf{2)} If $FG$ is constant, then by (\ref{f2.3}) we have $$8KFG+G|\nabla\phi_{+}|^{2}+F|\nabla\phi_{-}|^{2}+2\langle\nabla F, \nabla G\rangle\equiv0.$$
Let $p$ be a critical point of $F$. Since $K>0$, we must have $FG(p)=0$. Hence $FG\equiv0$, $F\equiv0$ or $G\equiv0$

\textbf{3)} If $F+G$ is constant, by Proposition \ref{p2.3}, we have $$|\nabla\phi_{+}|^{2}+4(K-R_{1234})F
=-(|\nabla\phi_{-}|^{2}+4(K+R_{1234})G).$$ If $R_{1234}$ is nowhere zero, then $\Delta F\geq0$ (when $R_{1234}<0$) or $\Delta G\geq0$ (when $R_{1234}>0$). So either $F\equiv0$ or $G\equiv0$.

Otherwise we assume that $R_{1234}(p)=0$. So in a neighborhood $\mathcal {U}(p)$ we have $K-R_{1234}>0$ and $K+R_{1234}>0.$ We must have $F\equiv0$ and $G\equiv0$ in $\mathcal {U}(p)$. Hence $F\equiv0$ and $G\equiv0$.

\textbf{4)} If $F-G$ is constant, then $\nabla F=\nabla G$. It  also follows directly.
\end{proof}

\begin{corollary}
If one of the two conditions holds: 1) $\Delta F\leq 8 \underline{k}F$; 2) $\Delta G\leq 8 \underline{k}G$, then either $F\equiv0$ or $G\equiv0$. Where $\underline{k}$ is the minimal sectional curvature of $M$.
\end{corollary}

\begin{proof}
Because $$G\Delta F+F\Delta G=G|\nabla\phi_{+}|^{2}+F|\nabla\phi_{-}|^{2}+8KFG\geq 16\underline{k}FG$$ and $$\int_{M}G\Delta Fdv=\int_{M}F\Delta Gdv.$$
One gets $$\int_{M}G(\Delta F-8 \underline{k}F)dv=\int_{M}F(\Delta G-8 \underline{k}G)dv\geq0.$$ The conditions in the corollary force that $\Delta F=8\underline{k}F$ or $\Delta G=8 \underline{k}G$. This leads to $F\equiv0$ or $G\equiv0$.
\end{proof}

%%%%%%%%%%%%%%%%%%%%%%%%%%%%%%%%%%%%%%%%%%%%%%%%%%%%%%%%%%%%%%%%%%%%%%%%%%%%%%
\section{Kato type inequality}
In this section we discuss the connections between Kato inequality and Hopf conjecture. Recall that the classical Kato inequality says that if $s$ is a section of  a Riemannian vector bundle $E$ over $M$ with connection $\nabla$, then
 $$|\nabla s|^{2}\geq|d|s||^{2}$$ for $s(x)\neq0$. This is an easy consequence of the Schwarz inequality $$|s|\cdot|d|s||=\frac{1}{2}|d|s|^{2}|=|<\nabla s,s>|\leq |s|\cdot|\nabla s|.$$

For harmonic 2-forms, we have the following Seaman's (\cite{[S1]}) refined result.

 \begin{lemma}
 Let $M$ be a compact Riemannian 4-manifold. Let $\phi$ be a harmonic 2-form. Then
\begin{equation}\label{f4.1}
|\nabla\phi|^{2}\geq\frac{3}{2}|d|\phi||^{2}.
\end{equation}
 \end{lemma}

The proof below is a slight generalization of \cite{[S1]}, but the idea is the same.

\begin{proof}
(Sketch). Let $\exp(2f)=|\phi|^{k}$ at $\phi\neq0$. Then $f=\frac{k}{2}\log|\phi|$. We can deduce that
\begin{equation}
2(-\Delta f-|df|^{2})|\phi|^{2}=-k|\phi|\Delta|\phi|+(k-\frac{k^{2}}{2})|d|\phi||^{2}.
\end{equation}

Consider the new metric $g^{'}=\exp(2f)g$. Let $F(\phi)$ be the curvature term in the Bochner formula
\begin{equation}
\frac{1}{2}\Delta|\phi|^{2}-\langle\phi, \Delta\phi\rangle=|\nabla\phi|^{2}+F(\phi),
 \end{equation}
 $\widetilde{F(\phi)}$ be the corresponding curvature term under metric change. Then we can calculate that
\begin{equation}
\widetilde{F(\phi)}|\phi|^{3k}=F(\phi)+2(-\Delta f-|df|^{2})|\phi|^{2},
\end{equation}
Because a harmonic 2-form is still harmonic under a conformal transformation (c.f.\cite{[Be]}).  Combing (4.2), (4.3) and (4.4), one has
\begin{equation}
|\nabla \phi|^{2}-(2k-\frac{k^{2}}{2})|d|\phi||^{2}=-\frac{k-1}{2} \Delta|\phi|^{2}+|\phi|^{3k}(|\nabla^{'}\phi|'^{2}-\frac{1}{2}\Delta^{'}|\phi|^{2-2k}),
\end{equation}
where $\nabla', |\cdot|', \Delta'$ denote the corresponding connection, metric, Laplacian under the conformal change of the metric.

On the other hand, by the formula of conformal change (c.f.\cite{[Be]} page 59), one has
\begin{equation}
\Delta^{'}|\phi|^{2-2k}=|\phi|^{-k}\Delta|\phi|^{2-2k}+2(k-k^{2})|\phi|^{-3k}|d|\phi||^{2}.
\end{equation}
Substituting(4.6) into (4.5), we obtain
\begin{equation}
|\nabla \phi|^{2}-(\frac{k^{2}}{2}+k)|d|\phi||^{2}=\frac{1-k}{2} \Delta|\phi|^{2}-\frac{1}{2}|\phi|^{2k}\Delta|\phi|^{2-2k}+ |\phi|^{3k}|\nabla^{'}\phi|'^{2}.
\end{equation}
Substituting
\begin{equation}
\Delta|\phi|^{2k}=2k|\phi|^{2k-1}\Delta|\phi|+2k(2k-1)|\phi|^{2k-2}|d|\phi||^{2}
\end{equation}
into (4.7), we have
\begin{equation}\label{f4.9}
|\nabla \phi|^{2}+(\frac{3k^{2}}{2}-3k)|d|\phi||^{2}=|\phi|^{3k}|\nabla^{'}\phi|'^{2}
\end{equation}
Choose $k=1$, we obtain the lemma.
\end{proof}

Seaman used (\ref{f4.1}) to prove that a compact Riemannian 4-manifold with $0.1714\leq sec_{M}\leq1$ is definite. So $S^{2}\times S^{2}$  does not admit metric with $0.1714\leq sec_{M}\leq1$.

In \cite{[CGH]}, Calderbank, Gauduchon and Herzlich obtained the refined Kato inequalities for harmonic forms in all dimensions.

\begin{lemma}
In the unconditional case, Seaman's result is the best possible.
\end{lemma}

\begin{proof}
If $|\nabla\phi|^{2}\geq\alpha|d|\phi||^{2}$ ($\alpha>\frac{3}{2}$) for any harmonic 2-form, by formula (\ref{f4.9}),
$$|\nabla \phi|^{2}+(\frac{3k^{2}}{2}-3k)|d|\phi||^{2}=|\phi|^{3k}|\nabla^{'}\phi|'^{2}\geq|\phi|^{3k}\alpha|d|\phi|^{'}|^{2}=\alpha(1-k)^{2}|d|\phi||^{2}.$$
Let $k\rightarrow \infty$. This leads to $d|\phi|=0$.
\end{proof}

In \cite{[Bet1]} Bettiol considered so called biorthogonal sectional curvature $sec_{M}^{\perp}$ for Riemannian 4-manifold. For each plane $\sigma\subset T_{p}M$,  	
$$sec_{M}^{\perp}(\sigma)\coloneqq\frac{1}{2}(sec_{M}(\sigma)+sec_{M}(\sigma^{\perp})),$$
where $\sigma^{\perp}$ is the plane orthogonal to $\sigma$. He proved that $S^{2}\times S^{2}$ admits a metric with $sec_{M}^{\perp}>0$. Moreover, he showed in \cite{[Bet2]} that  for a compact simple-connected 4-manifold $M$, the following three statements  are equivalent: 1) $M$ admits $sec_{M}^{\perp}>0$; 2) $M$ admits $Ric_{M}>0$; 1) $M$ admits $scal_{M}>0$.

So Theorem 1.1 essentially holds for positive biorthogonal sectional curvature. By Bettiol's result and Lemma 4.2, we know that only Theorem 2.1 is not enough for Hopf conjecture. We need more information about how sectional curvature affects harmonic 2-forms.

We still can ask the following question:

{\em   Let $\phi$ be a harmonic 2-form on a compact Riemannian 4-manifold with \emph{positive sectional curvature}. Does $|\nabla\phi|^{2}\geq2|d|\phi||^{2}$ hold?}

 If the answer is positive, then Hopf conjecture is true.

\section{Appendix}
Though the computation of formula (\ref{f2.2}) is standard and may be found elsewhere, for convenience we write down the details.

Let $\phi=\sum_{k,l}\phi_{kl}\omega^{k}\wedge\omega^{l}$ be a harmonic 2-form. $\phi_{kl}=-\phi_{lk}$

\begin{lemma}
$\Delta\phi_{kl}(p)=\sum_{p}(Ric_{kp}\varphi_{pl}+Ric_{lp}\varphi_{kp})-2\sum_{p,q}R_{kplq}\varphi_{pq}$.
\end{lemma}

\begin{proof}
It is easy to see that $$\nabla^{2}_{e_{i}e_{i}}\varphi=(\Delta\varphi_{kl})\omega^{k}\wedge\omega^{l}+\varphi_{kl}\nabla_{e_{i}}\nabla_{e_{i}}(\omega^{k}\wedge\omega^{l}).$$
Then
\begin{eqnarray*}
2\Delta\varphi_{kl}&=&\langle\nabla^{2}_{e_{i}e_{i}}\varphi, \omega^{k}\wedge\omega^{l}\rangle\\
                   &=&\langle\omega^{i}\wedge i(e_{j})R_{e_{i}e_{j}}\varphi, \omega^{k}\wedge\omega^{l}\rangle\\
                   &=&\phi_{pq}\langle\omega^{i}\wedge i(e_{j})R_{e_{i}e_{j}}(\omega^{p}\wedge\omega^{q}), \omega^{k}\wedge\omega^{l}\rangle\\
                   &=&\phi_{pq}\langle\omega^{i}\wedge i(e_{j})[(R_{e_{i}e_{j}}\omega^{p})\wedge\omega^{q})+\omega^{p}\wedge (R_{e_{i}e_{j}}\omega^{q})], \omega^{k}\wedge\omega^{l}\rangle.
\end{eqnarray*}
Note that $R_{e_{i}e_{j}}\omega^{p}=\sum_{m}R_{ijpm}\omega^{m}$ and $i(e_{j})R_{e_{i}e_{j}}\omega^{p}=R_{ijpj}$.
Hence
\begin{eqnarray*}
2\Delta\varphi_{kl}&=&\sum_{i,j,p,q}\phi_{pq}\langle\omega^{i}\wedge[R_{ijpj}\omega^{q}-\delta^{q}_{j}\sum_{m}R_{ijpm}\omega^{m}\\
                   &&+\delta^{p}_{j}\sum_{m}R_{ijqm}\omega^{m}-R_{ijqj}\omega^{p}],\omega^{k}\wedge\omega^{l}\rangle\\
                   &=&\sum_{i,j,p,q}\langle\phi_{pq}R_{ijpj}\omega^{i}\wedge\omega^{q}-\phi_{pq}R_{ijqj}\omega^{i}\wedge\omega^{p},\omega^{k}\wedge\omega^{l}\rangle\\
                   &&+\sum_{i,j,p,q,m}\langle\phi_{pq}\delta^{p}_{j}R_{ijqm}\omega^{i}\wedge\omega^{m}-\phi_{pq}\delta^{q}_{j}R_{ijpm}\omega^{i}\wedge\omega^{m},\omega^{k}\wedge\omega^{l}\rangle\\
                   &=&\sum_{j,p}(\phi_{pl}R_{kjpj}-\phi_{pk}R_{ljpj})-\sum_{j,q}(\phi_{lq}R_{kjqj}-\phi_{kq}R_{ljqj})\\
                   &&+\sum_{j,p,q}(\phi_{pq}\delta^{p}_{j}R_{kjql}-\phi_{pq}\delta^{p}_{j}R_{ljqk})-
                   \sum_{j,p,q}(\phi_{pq}\delta^{q}_{j}R_{kjpl}-\phi_{pq}\delta^{q}_{j}R_{ljpk})\\
                   &=&\sum_{p}(\phi_{pl}Ric_{kp}-\phi_{pk}Ric_{lp})-\sum_{q}(\phi_{lq}Ric_{kq}-\phi_{kq}Ric_{lq})\\
                   &&+\sum_{p,q}(\phi_{pq}R_{kpql}-\phi_{pq}R_{lpqk})-\sum_{p,q}(\phi_{pq}R_{kqpl}-\phi_{pq}R_{lqpk})\\
                   &=&2\sum_{p}(Ric_{kp}\varphi_{pl}+Ric_{lp}\varphi_{kp})-4\sum_{p,q}R_{kplq}\varphi_{pq}.
\end{eqnarray*}

\end{proof}

\bibliographystyle{amsplain}

\end{document}